\documentclass[10pt]{amsart}
\usepackage{amsmath}
\usepackage{amsfonts}
\usepackage{amssymb}
\usepackage{amsthm}
\usepackage{url}
\usepackage{dsfont}
\usepackage{graphicx}
\usepackage{caption}
\usepackage{subcaption}
\usepackage{comment} 
\usepackage{stmaryrd}
\usepackage{hyperref}
\usepackage{todonotes}
\usepackage{color}
\usepackage{enumerate}
\usepackage{stmaryrd}
\usepackage{xypic}
\usepackage{tikz-cd}
\usepackage[english,french]{babel}

\newcommand{\git}{\mathbin{
  \mathchoice{/\mkern-6mu/}
    {/\mkern-6mu/}
    {/\mkern-5mu/}
    {/\mkern-5mu/}}}

\numberwithin{equation}{section}

\newtheorem{proposition}{Proposition}[section]
\newtheorem{lemma}[proposition]{Lemma}
\newtheorem{theorem}[proposition]{Theorem}
\newtheorem{corollary}[proposition]{Corollary}

\theoremstyle{definition}
\newtheorem{remark}[proposition]{Remark}
\newtheorem{definition}[proposition]{Definition}
\newtheorem{example}[proposition]{Example}

\DeclareMathOperator{\id}{id}

\DeclareMathOperator{\Aut}{Aut}
\DeclareMathOperator{\Ric}{Ric}

\DeclareMathOperator{\inte}{int}
\DeclareMathOperator{\CM}{CM}
\DeclareMathOperator{\Ima}{Im}
\DeclareMathOperator{\red}{red}

\newcommand{\R}{\mathbb{R}}
\newcommand{\C}{\mathbb{C}}

\newcommand{\Q}{\mathbb{Q}}

\newcommand{\pr}{\mathbb{P}}
\renewcommand{\epsilon}{\varepsilon}
\newcommand{\scH}{\mathcal{H}}

\newcommand{\M}{\mathcal{M}}

\renewcommand{\O}{\mathcal{O}}

\newcommand{\ddb}{i\partial \bar\partial}

\newcommand{\mfh}{\mathfrak{h}}

\newcommand{\F}{\mathcal{F}}
\newcommand{\B}{\mathcal{B}}

\renewcommand{\L}{\mathcal{L}}
\newcommand{\J}{\mathcal{J}}
\newcommand{\db}{\bar\partial}
\renewcommand{\H}{\mathcal{H}}
\renewcommand{\M}{\mathcal{M}}
\newcommand{\X}{\mathcal{X}}
\newcommand{\Y}{\mathcal{Y}}

\newcommand{\scK}{\mathcal{K}}

\newcommand{\newabstract}[1]{%
  \par\bigskip
  \csname otherlanguage*\endcsname{#1}%
  \csname captions#1\endcsname
  \item[\hskip\labelsep\scshape\abstractname.]
}

\title[Moduli of polarised manifolds via canonical K\"ahler metrics]{Moduli of polarised manifolds via canonical K\"ahler metrics}

\author[Ruadha\'i Dervan and Philipp Naumann]{Ruadha\'i Dervan and Philipp Naumann}

\address{Ruadha\'i Dervan, DPMMS, Centre for Mathematical Sciences, Wilberforce Road, Cambridge CB3 0WB, United Kingdom}
\email{Ruadhai.Dervan@glasgow.ac.uk}

\address{Philipp Naumann, Universit\"at Bayreuth, Mathematisches Institut, Universit\"atsstra{\ss}e 30, 95440 Bayreuth, Germany}
\email{philipp.naumann@uni-bayreuth.de }

\begin{document}

\selectlanguage{english}
\begin{abstract} We construct a moduli space of  polarised manifolds which admit a constant scalar curvature K\"ahler metric. We show that this space admits a natural K\"ahler metric.

\end{abstract}

\selectlanguage{english}

\maketitle

\section{Introduction}

One of the main goals of algebraic geometry is to construct moduli spaces of polarised varieties, that is, varieties endowed with an ample line bundle. It has  long been understood that in order to obtain a reasonably well behaved moduli space, one must impose a stability condition. This was the primary motivation for Mumford's Geometric Invariant Theory (GIT) \cite{GIT}. While this was successful for curves, it is now known that GIT techniques essentially fail to produce reasonable moduli spaces of polarised varieties in higher dimensions \cite{WX,NSB}. One way to attempt to rectify this is to replace GIT stability with K-stability, as defined by Tian and Donaldson \cite{GT,SD4}. While K-stability is motivated by GIT, it is not a genuine GIT notion, meaning that constructing moduli spaces of K-polystable varieties by algebro-geometric techniques is out of reach at present in general. 

Here we take an analytic approach, under the assumption that the polarised variety is actually smooth. In this case K-polystability is conjectured by Yau, Tian and Donaldson to be equivalent to the existence of a constant scalar curvature K\"ahler (cscK) metric on the polarised manifold \cite{STY,GT,SD4}. Thus the existence of a cscK metric can be seen as a sort-of analytic stability condition.  This perspective allows us in addition to consider non-projective K\"ahler manifolds, where by a polarisation we shall mean a K\"ahler class, and where there is still an analogue of the Yau--Tian--Donaldson conjecture \cite{DRoss,RD,ZSD,ZSD2, DXZ, PMP}. Our main result is the following:

\begin{theorem}\label{mainthm} There exists a Hausdorff complex space which is a moduli space of polarised manifolds which admit a cscK metric. \end{theorem}

The moduli space satisfies the properties that its points are in correspondence with cscK polarised manifolds, and any family of cscK polarised manifolds induces a map to the moduli space which is compatible with this correspondence. While the moduli space is disconnected, in practice, one is typically interested in a single connected component containing a polarised manifold of interest, and our proof actually  constructs the moduli space component-by-component.

This construction provides an \emph{analytic} solution to the moduli problem for polarised manifolds, and hence in particular smooth polarised varieties, constructing a moduli space as a Hausdorff complex space. We should say immediately that the novelty in our construction is solely in the case of polarised manifolds admitting continuous automorphisms. In the case of discrete automorphism group, the existence of such a moduli space is an important result of Fujiki--Schumacher \cite{FS}. As is well understood in moduli problems, the case when the objects admit continuous automorphisms presents new challenges, as polystability is not an open condition, even in the analytic topology. This is reflected for us by the fact that the existence of a cscK metric is not an open condition in families of polarised manifolds \cite{SD5,GT}. We thus take quite a different approach to the construction, which can be sketched as follows.

We begin with a fixed polarised manifold and consider its Kuranishi space. The deformation theory of cscK manifolds due to Ortu, Br\"onnle and Sz\'ekelyhidi then gives an understanding of which deformations admit a cscK metric in terms of a local GIT condition \cite{ortu,TB,GS}. We give a more thorough understanding of the orbits of this action and their stabilisers, using two fundamental results: firstly, the uniqueness of cscK metrics on a polarised manifold up to automorphisms \cite{SD2,BB}; secondly, the work of Inoue on automorphism groups of deformations \cite{EI}. This allows us to construct a local moduli space around the polarised manifold under consideration by using Heinzner--Loose's theory of GIT quotients of complex spaces \cite{HL}. We then show that these local moduli spaces can be glued; the proof of this crucially uses Chen--Sun's deep work on the uniqueness of cscK degenerations \cite{CS}.

An important aspect   of the moduli space that we establish is its Hausdorffness: one of the basic reasons for using stability conditions in moduli theory is to create Hausdorff moduli spaces, as Hausforffness fails for arbitrary families of polarised manifolds, even in simple cases such as Fano manifolds or projective bundles. This property implies various results around the behaviour of cscK manifolds in families, which we state as the following corollary of Theorem \ref{mainthm}. To explain the statement, define a polarised manifold to be \emph{analytically K-semistable} if it admits a degeneration to cscK manifold (in the sense of Definition \ref{def:semistable}). Analytically K-semistable manifolds are an analytic analogue of K-semistable varieties.

\begin{corollary}\label{intro-degen}

Let $C$ be a Riemann surface, with $0 \in C$. Suppose $$\xymatrix@C=14pt@R=16pt{
&(\Y,\alpha_{\Y}) \ar[dr]&& \mathcal (\tilde \Y,\alpha_{\tilde \Y}) \ar[dl] \\
 &   & C }
$$ are two polarised families such that $({\Y},\alpha_{\Y})|_{C\backslash\{0\}} \cong  (\tilde \Y,\alpha_{\tilde \Y})|_{C\backslash\{0\}}.$

\begin{enumerate}[(i)]
\item If  $(\Y_0,\alpha_{\Y_0})$ and $(\tilde \Y_0,\alpha_{\tilde \Y_0})$ are analytically K-semistable, then they degenerate to the same cscK manifold. 
\item  If  $(\Y_0,\alpha_{\Y_0})$ and $(\tilde \Y_0,\alpha_{\tilde \Y_0})$ admit a cscK metric, then $(\Y_0,\alpha_{\Y_0})$ and $(\tilde \Y_0,\alpha_{\tilde \Y_0})$ are isomorphic.

\end{enumerate}

\end{corollary} \noindent Part $(ii)$ gives a version of the valuative criterion for separatedness for the moduli space of cscK manifolds. In the Fano K\"ahler--Einstein setting, this is an important result due to Spotti--Sun--Yao \cite{SSY} and Li--Wang--Xu \cite{LWX}.

Once one has constructed a moduli space, natural questions surrounding its geometry arise. Our next result shows that the moduli space is actually a \emph{K\"ahler} space. Denote by $\M$ the moduli space of cscK manifolds.

\begin{theorem}\label{WP} $\M$ admits a natural Weil--Petersson type K\"ahler metric $\eta$.  On the moduli space of projective cscK manifolds, there is a natural $\Q$-line bundle $\L\to\M$. In this case, $\eta$ is the curvature of a Hermitian metric on $\L$. \end{theorem}

It follows that any compact complex subspace of $\M$ is projective, when we consider the moduli space of cscK manifolds where the polarisation satisfies $\alpha=c_1(L)$ for $L$ ample, so that the manifolds being parametrised are projective. What we establish is that both the CM line bundle and the Weil--Petersson metric defined on the Kuranishi space actually descend to the moduli space $\M$ via our local quotient construction, so it is natural to call $\eta$ the Weil--Petersson metric on $\M$ and $\L$ the CM line bundle. The above is due to Li--Wang--Xu in the Fano K\"ahler--Einstein case \cite{LWX2} and Fujiki--Schumacher in the case of discrete automorphism group \cite{FS}. We also establish some basic properties of the CM line bundle and the Weil--Petersson metric in relation to families of cscK manifolds, which we expect to be useful in applications.

Typically in moduli problems, in the presence of automorphisms one expects a moduli space to naturally be a stack (in the category of complex spaces), and in particular to be a version of a good moduli space in the sense of Alper \cite{JA}. We mention that our construction does have more structure than we have described above: it is essentially an analytic moduli space in the sense of Odaka--Spotti--Sun \cite[Definition 3.14]{OSS}, which is a condition on the moduli space locally having the structure of a quotient of a deformation space.  

A natural and important question is whether $\M$ admits a modular compactification, notably in the setting that the manifolds under consideration are projective and we take our K\"ahler classes to be induced by ample line bundles. In this case, under further assumptions on  the geometry of the polarised varieties being parametrised, it may be possible to compactify by including singular K-polystable varieties at the boundary. If this were possible, one would expect the CM line bundle to extend to the compactification, and hence one would further expect the compactification to be projective. This would be one route to showing that $\M$ admits an algebraic structure, which is an important question inaccessible using the techniques we employ in the present work. It is, in general, unrealistic to expect this strategy to be successful without assumptions on the underlying geometry of the polarised varieties, and it may be instead that one could compactify by using something more general than varieties, perhaps using a class of schemes of infinite type, by analogy with the non-finitely generated filtrations of Witt Nystr\"om and Sz\'ekelyhidi \cite{DWN,GS-filtrations}. It is also natural to ask whether these ideas have analogue in the non-projective setting (see Berman \cite[Section 5]{RB-MA} for a related discussion).

Lastly, we remark that throughout we work with \emph{reduced} complex spaces. The reason for this is the lack of literature on analytic GIT for non-reduced complex spaces. This forces us to work with the induced reduced complex space of the Kuranishi space, and means that for families of polarised manifolds over a \emph{non-reduced} base $B$, we do not obtain a map to the moduli space. Instead if $B_{\red}\hookrightarrow B$ is the reduction, then after pulling back the family to $B_{\red}$, we obtain a natural map $B_{\red} \to \M$. With further work on analytic GIT, we expect that our techniques would give a slightly more natural structure of a complex space on the moduli space, such that the moduli space we construct in Theorem \ref{mainthm} is the induced reduced complex space. This would be important, for example, in proving that $\M$ satisfies a universal property. 

\subsection*{Moduli of K-polystable varieties} A solution to the Yau--Tian--Donadson conjecture, predicting that K-polystability is equivalent to the existence of a cscK metric, would imply that we have constructed a moduli space of smooth K-polystable polarised varieties as a complex space. While the techniques in the present work say little about this algebraic problem (and in particular our construction produces a moduli space as a complex space rather than an algebraic space), we mention that such moduli constructions have been achieved in three important cases. 

The first is in the case that the polarisation (i.e. ample line bundle) $L=K_X$ is the canonical class. Then K-stability of canonically polarised varieties is equivalent to the variety having so-called semi-log canonical singularities through work of Odaka \cite{YO2,YO3}, and the construction of the corresponding moduli space (called the KSBA moduli space) is one of the major successes of the minimal model program \cite{jk-moduli-survey}, generalising the moduli space of curves $\overline{\M_g}$. We note that (even mildly singular) canonically polarised varieties admit K\"ahler--Einstein metrics, which are therefore cscK when the variety is smooth \cite{BG,STY2,TA}. Similarly, Schumacher and Viehweg have constructed a moduli space of smooth polarised Calabi--Yau manifolds \cite{GSchum,EV}, which by Odaka \cite{YO2} are all K-stable and which by Yau's solution of the Calabi conjecture admit K\"ahler--Einstein metrics \cite{STY2}. Thus the classical moduli spaces are all moduli of K-stable varieties, and our Theorem \ref{mainthm} gives another construction of these spaces.

The third case is the Fano case, so that the polarisation $L=-K_X$ is anticanonical. In this case, the equivalence between K-polystability and the existence of a K\"ahler-Einstein metric is the content of Chen--Donaldson--Sun's solution of the Yau--Tian--Donaldson conjecture for Fano manifolds \cite{CDS-jams,GT,RB}. Using many of the deep results of \cite{CDS-jams}, Li--Wang--Xu and Odaka have given a construction of a moduli space of $\Q$-smoothable K-polystable Fano varieties \cite{LWX,YO}, which uses largely different, more algebraic techniques to our approach. It is also worth remarking that there has been very recent progress in constructing this moduli space purely algebro-geometrically, particularly towards the construction of moduli of uniformly K-stable Fano varieties in the sense of \cite{RD2,BHJ}, and we refer to \cite{BL,BX,CZ} for important results in this direction. 

In both the canonically polarised case and the Fano case, we emphasise once more that the moduli spaces constructed using algebraic techniques satisfy significantly stronger conditions than what we obtain analytically, for example they are both known to be \emph{algebraic} spaces which are quasi-projective, and even projective in the KSBA case \cite{LWX2,OF,JK,EV}, while our analytic construction produces a \emph{complex} space. 

\begin{remark} While this work was in progress, we learned of work of Eiji Inoue, who constructs a moduli space of Fano manifolds which admit a K\"ahler--Ricci soliton \cite{EI}. We crucially use an argument of Inoue to understand the automorphism groups of deformations of cscK manifolds, and his ideas are important in our work. The major difference in the basic approach is in how we give the natural topological moduli space the structure of a complex space:  our argument is more direct, constructing and gluing charts explicitly, while Inoue develops an analytic theory of stacks. The main tools are similar (for example, uniqueness of canonical metrics and uniqueness of degenerations of semistable objects), however, and so the approaches are most likely equivalent. We warmly thank Eiji Inoue for many helpful discussions on this and related topics. \end{remark}

\vspace{4mm} \noindent {\bf Outline:} We discuss the aspects of analytic GIT and cscK metrics which we require in Section \ref{sec:prelim}. In Section \ref{section:construction}, we construct the moduli space of cscK manifolds, establishing Theorem \ref{mainthm} and Corollary \ref{intro-degen}. Lastly, the content of  Section \ref{sec-kahler} is our material on the Weil--Petersson metric and the CM line bundle, and in particular this section contains the proof of Theorem \ref{WP}

\vspace{4mm} \noindent {\bf Notation:} Let $X$ be a reduced complex space. We briefly recall the notion of a K\"ahler metric on $X$, following \cite{HHL}. If $h: \Delta \to X$ is a holomorphic map from the unit disc and $\phi$ is a continuous real valued function, we say that $\phi$ satisfies the \emph{mean value condition} if $$\phi \circ h(0) \leq \frac{1}{2\pi}\int^{2\pi}_0 (\phi \circ h)(re^{i\theta})d\theta \textrm{ for } 0<r<1.$$ If this is the case for all such $h$, we say that $\phi$ is \emph{plurisubharmonic}. $\phi$ is in addition said to be \emph{strictly plurisubharmonic} if for every $x \in X$ and smooth $f$ defined in a neighbourhood of $x$, there is an $\epsilon>0$ such that $\phi+\epsilon f$ is plurisubharmonic in a (perhaps smaller) neighbourhood of $x$. Here a function is called \emph{smooth} if in a local chart $U \subset X$ biholomorphic to an open subset of the form $V(I) \subset \C^m$ (with $V(I)$ the vanishing locus of a set of holomorphic functions), it is the restriction of a smooth function on $\C^m$ \cite[p.\ 45]{PL}.

A \emph{K\"ahler metric} on $X$ is a collection of strictly plurisubharmonic functions $(\phi_{\alpha})$ on a covering $(U_{\alpha})$ of $X$ such that the differences $\phi_{\alpha} - \phi_{\beta}$ on $U_{\alpha} \cap U_{\beta}$ are harmonic, in the sense that they are the imaginary part of a holomorphic function on $U_{\alpha} \cap U_{\beta}.$  This data will also be written $\omega$, as is typical in the case of a smooth manifold. 

We will typically have a complex space $X$ which has a natural stratification by locally closed complex manifolds. In this situation on each stratum the $\phi_{\alpha}$ will either be smooth, or may be taken to be of class $C^l$ for some large $l$ (this $l$ will typically arise from some Sobolev embedding). Thus on each stratum $\omega = i\partial\bar\partial\phi_{\alpha}>0$, producing a perhaps mildly singular K\"ahler metric in the usual sense. 

We refer to \cite{ABB,HL,DG} for an introduction to GIT on complex spaces, and to \cite{GS2} for an introduction to cscK metrics. 

\vspace{4mm} \noindent {\bf Acknowledgements:} We thank Daniel Greb, Yue Fan, Eiji Inoue, Yuji Odaka, Cristiano Spotti, G\'abor Sz\'ekelyhidi and David Witt Nystr\"om for helpful discussions, and the referee for several useful suggestions. RD received funding from the ANR grant ``GRACK'' during part of this work. PN received funding from the ERC grant ``ALKAGE''.

\section{Preliminaries}\label{sec:prelim}
\subsection{Analytic geometric invariant theory}  Let $Z$ be a reduced Stein space with an action of a reductive group $G$. We are interested in taking a quotient of $Z$ by $G$; as usual the quotient will only represent \emph{polystable} orbits. A good summary of this theory, namely that of analytic \emph{geometric invariant theory (GIT)}, can be found in \cite[Section 10]{ABB}.

\begin{definition}\cite{GIT} We say that $p \in Z$ is 
\begin{enumerate}[(i)]
\item \emph{polystable} if the orbit $G.p$ is closed,
\item \emph{semistable} otherwise. 
\end{enumerate} 

 \end{definition}

The notion of a quotient we shall use is the following standard variant of the algebraic notion.

\begin{definition}\cite[Definition 1.1]{HL} We say that $Z \to W$ is a \emph{quotient} if
\begin{enumerate}[(i)]
\item for any open $V \subset W$, the inclusion $\O(V) \hookrightarrow \O(\pi^{-1}(U))^G$ is an isomorphism,
\item the polystable orbits are in bijection with the points in $W$.
\end{enumerate}
 \end{definition}
 
The second condition identifies the points of the quotient, while the first gives the complex structure. Once a quotient exists, it is unique up to isomorphism. The following result due to Snow and Heinzner then constructs the quotient, which we call the \emph{GIT quotient}.

\begin{theorem}\cite{DS,PH}  The quotient $Z\git G$ of $Z$ by $G$ exists. \end{theorem}

\subsection{Constant scalar curvature K\"ahler metrics} We review some aspects of the theory of constant scalar curvature K\"ahler metrics that we shall require. Throughout we let $X$ be a compact complex manifold with a K\"ahler class $\alpha$; we call this data a \emph{polarised manifold}. A \emph{family} of complex manifolds $f: \X\to S$ means a proper smooth morphism of reduced complex spaces with connected fibres. A \emph{polarised family} is a family equipped with a \emph{polarisation}, by which we mean a section $\alpha_{\X} \in \Gamma(S,R^2f_*\R)$ which induces a polarisation $\alpha_{\X_s}$ in the usual sense on all fibres $\X_s$ \cite[Definition 3.1]{FS}.

\begin{definition}The \emph{scalar curvature} of a K\"ahler metric $\omega$ on $X$ is defined to be the contraction of the Ricci curvature $$S(\omega) = \Lambda_{\omega} \Ric\omega,$$ and we say that $\omega$ is a \emph{constant scalar curvature K\"ahler} (cscK) metric if $S(\omega)$ is constant. \end{definition}

The first main link between the geometry of $X$ and the cscK condition is the following, due to Berman--Berndtsson \cite{BB}, Calabi \cite{EC}, Chen \cite{XC}, Donaldson \cite{SD2}, Lichnerowicz \cite{AL} and Matsushima \cite{YM}. 

\begin{theorem}\label{thm:uniqueness} Fix a K\"ahler class  $\alpha$ on $X$, and suppose $\omega \in \alpha$ is a cscK metric. Then:
\begin{enumerate}[(i)]
\item The automorphism group $\Aut(X,\alpha)$ is reductive, and is the complexification of the isometry group of $(X,\omega)$.
\item If $\omega' \in \alpha$ is another cscK metric, then there is a $g \in \Aut(X,\alpha)$ such that $\omega = g^*\omega'$. 
\end{enumerate}
\end{theorem}

Here the automorphism group $\Aut(X,\alpha)$ is characterised as a subgroup of $\Aut(X)$ by its Lie algebra, which consists of vector fields which vanish somewhere \cite{LS}. We will denote $G_X = \Aut(X,\alpha)$, which is reductive, and is the complexification of the isometry group of $(X,\omega)$, which we denote $H_X\subset G_X$. In addition we denote by $\mfh_X$ the Lie algebra of $H_X$. In the projective case, by which we mean the case that $\alpha = c_1(L)$ for an ample line bundle $L$, the automorphism group consists of automorphisms of $X$ which lift to $L$.

We will also require the following.  To state the result, let us say that a polarised manifold $(Y,\alpha_Y)$ admits a \emph{degeneration} to $(X,\alpha_X)$ if there is a polarised family $(\X,\alpha_{\X}) \to \C$ with general fibres all isomorphic to $(Y,\alpha_Y)$ and central fibre isomorphic to $(X,\alpha_X)$.

\begin{theorem}[Chen--Sun]\cite[Theorem 1.6]{CS}\label{thm:CS} If $(Y,\alpha_Y)$ admits a degeneration to two cscK manifolds $(X,\alpha_X)$ and $(Z,\alpha_Z)$, then $(X,\alpha_X) \cong (Z,\alpha_Z)$.
\end{theorem}

The Chen--Sun  statement requires that the K\"ahler class be integral; the assumption is only used so that the class $\alpha_{\X_s} \in H^2(\X_s,\R)$ can be taken to be independent of $s$ after fixing the underlying smooth structure of the families \cite[Proof of Theorem 1.6]{CS}. This is automatic for polarised families of complex manifolds  (see for example \cite[Theorem 8.2]{AF}), giving the above, as raised by Chen--Sun \cite[Remark (4), Section 9]{CS}.

\subsection{Deformation theory}\label{sec:deformations} We recall the deformation theory of cscK manifolds, established by  Ortu, Br\"onnle and Sz\'ekelyhidi  \cite[Appendix A]{ortu}\cite{TB,GS}; our discussion follows Ortu's work. The results can be viewed as providing an analogue of Luna's slice theorem in GIT. We begin with some notation. We fix a compact symplectic manifold $(M,\omega)$, and let $\J$ denote the space of almost complex structures on $M$ which are compatible with $\omega$. The space $\J$ admits an action of the group $\H$ of exact symplectomorphisms of $(M,\omega)$.  We will also be interested in the subset $\J^{\inte}\subset \J$ consisting, by definition, of integrable almost complex structures, which is preserved by $\H$, and where  the $\H$-stabiliser of any complex structure is simply the group of Hamiltonian isometries of the K\"ahler metric (induced by) $\omega$ with respect to the complex structure \cite[Section 6.1]{GS2}.

We assume that we have a polarised manifold $(X,\alpha)$ inducing a $J_0 \in \J$ with respect to which $\omega$ is a cscK metric. To understands its deformations, we consider as in Sz\'ekelyhidi \cite{GS} the elliptic complex $$C^{\infty}_0(X) \overset{P}{\to}T_{J_0}\J\overset{\db}{\to}\Omega^{0,2}(T^{1,0}).$$ Here $T_{J_0}\J$ consists of the $\beta \in \Omega^{0,1}(T^{1,0})$ satisfying $$\omega(\beta(v),w) + \omega(v,\beta(w))=0.$$ The space $C^{\infty}_0(X)$ denotes the functions with integral zero over $X$ and the operator $P: C^{\infty}_0(X){\to}T_{J_0}\J$ is defined by $P(h) = \db v_h$, where $v_h$ is the Hamiltonian vector field associated with $h$. We denote by $$\tilde H^1 = \{\beta \in T_{J_0}\J\ |\ P^*\beta = \db \beta = 0 \},$$ which is the kernel of the elliptic operator $P^*P+ \db^*\db$ and hence finite dimensional. 

The group $H_X$ of isometries of $(X,\omega)$ acts naturally on $\Omega^{0,1}(T^{1,0})$ by pulling back sections, and this induces a linear action of $H_X$ on  $\tilde H^1 \subset \Omega^{0,1}(T^{1,0})$. As $G_X$ is the complexification of $H_X$, this induces a linear action of $G_X$ on $\tilde H^1$.

A version of Kuranishi's theorem, explained by Sz\'ekelyhidi \cite[Proposition 7]{GS}, produces a $K$-equivariant holomorphic map $$\Psi: \hat B \to \J$$ from a ball $\hat B \subset \tilde H^1$ around the origin, mapping $0\in \hat B$ to $J_0$. The space $\hat B$ admits a universal family $\pi: (\X,\alpha_{\X}) \to \hat B$, such that $\X$ is an almost complex manifold with almost complex structure $J_{\X}$ and $\pi$ is a holomorphic submersion between almost complex manifolds. In fact, smoothly $\X = \hat B \times M$ and the almost complex structure on $\X$ is constructed such that a fibre $\X_p = \pi^{-1}(p)$ is given by $(M,\Psi(p))$,  namely $M$ endowed with the almost complex structure $\Psi(p)$. Throughout we will frequently shrink $\hat B$, by which we simply mean taking a smaller neighbourhood of the origin.

The $H_X$-action on $\hat B$ lifts in addition to one on $\X$ making $\pi$ a $H_X$-equivariant morphism, and $\X$ admits a relatively almost K\"ahler metric (i.e. almost K\"ahler on each fibre) $ \tilde \omega_{\X}$ induced by $\omega$. More precisely,  $\omega$ is pulled back from $M$ through the smooth identification $\X = \hat B \times M$ to produce what we denote $\tilde \omega_{\X}$. 

We will mostly be interested in the restriction of the universal family to the integrable locus  which we denote $$B\subset \hat B,$$ and which may be singular. We note the integrable locus is preserved by the $G_X$-action (in the sense that  if $p\in B$ and $g.p\in \hat B$, then $g.p \in B$). We remark that since we work in a neighbourhood of the origin in $\tilde H^1$, we do not have a genuine $G_X$-action on $B$, most formally the ``local'' $G_X$-action gives a pseudogroup structure to $B$ in the sense of Cartan. As $B$ may not be reduced, we replace it with its reduction to later apply Stein GIT. 

On the integrable locus $B$, the Kuranishi family $\X\to B$ is a \emph{complete} deformation space: given any $p\in B$ and any family $(\Y,\alpha_Y)\to B_{\Y}$ with $b\in B_{\Y}$ such that $(\Y_b,\alpha_{\Y_b}) \cong (\X_p,\alpha_p)$, there is a morphism $\sigma$ from a neighbourhood of $b$ in $B_{\Y}$ to $B$ such that $\sigma(p)=b$ and such that $\sigma$ pulls back $(\X,\alpha_{\X})$ to $(\Y,\alpha_{\Y})$. We also note that if $q\in G_X.p$, then from the constructions $(\X_p, \alpha_{\X_p}) \cong(\X_q, \alpha_{\X_q})$.

We will be interested in modifying the initial relative almost K\"ahler form $\tilde \omega_{\X}$ on $\X$. Recalling that smoothly $\X = \hat B \times M$, we may modify the structure by changing $\omega_{\X}$ to $$\omega_{\X} =  \tilde \omega_{\X} +dJ_{\X}d\phi \in \alpha_{\X},$$ for $\phi$ a function on $\X =\hat B_{\epsilon}\times M$, which we think of a family of functions on $M$ parametrised by $p\in \hat B$. On the integrable locus, the resulting form  $ \omega_{\X} $ is a closed, $J_{\X}$-invariant $2$-form on $X$ which, provided $\phi$ is sufficiently small (or after shrinking the base) is relatively K\"ahler (note that as $J_{\X}$ is not integrable in general, $\omega_{\X}$ is not necessarily $J_{\X}$-compatible). The function $\phi$ is constructed on all of $\X$, but we will only be interested in its geometric properties on the integrable locus.  We set $$\omega_{\X_p} = (\omega_{\X} +dJ_{\X}d\phi)|_{\X_p},$$ recalling $\X_p \cong M$ smoothly. Fixing a reference metric on $M$ defines a Sobolev space $L^2_{k+2}(M)$ of functions on $M$, and we will be interested in $\phi$ with $\phi|_{\X_p}\in L^2_{k+2}(M)$ varying smoothly with $p\in \hat B$.

\begin{theorem}\label{ortu}\cite[Theorem A.1]{ortu} Perhaps after shrinking $\hat B$, there exists a function $\phi$ on $\X$ inducing $$\omega_{\X} =  \tilde \omega_{\X} +dJ_{\X}d$$ such that $\phi|_{\X_p} \in L^2_{k+2}(M)$ for all $p\in \hat B_{\epsilon}$, varying smoothly with $p$, and an open set $\hat C \subset \hat B$ such that the following holds for $p \in C$ in the integrable locus $C\subset \hat C$:
\begin{enumerate}[(i)]
\item if the $G_X$-orbit of $p$ is closed in $ B$, then there is a point $q\in G.p \cap B$ such that the K\"ahler manifold $(\X_q,\omega_{\X_q})$ is cscK;
\item if the $G_X$-orbit of $p$ is not closed in $ B$, then the closure of the $G_X$-orbit of $p$ contains a point $q \in B$ with $(\X_q,\omega_{\X_q})$ cscK;
\item if the manifold $(X_p,\alpha_{\X_p})$ associated to $p$ admits a cscK metric and the $G_X$-orbit of $p$ is not closed in $ C$, then the manifold associated to $q$ described in $(ii)$ is isomorphic to $(\X_p,\alpha_{\X_p})$.
\end{enumerate} 
\end{theorem}

A main point of Ortu's approach to proving this is to change the K\"ahler structure on $\X$, as opposed to changing the almost complex structure and hence the map $\Psi$. It will be more useful in our subsequent results, however, to change instead the almost complex structure. Moser's trick explains the relationship between the two approaches, as described for example by Sz\'ekelyhidi \cite[Section 3]{GS}; we recall the construction.

Fixing $p\in \hat B$, we write $$\omega_{p,t} = \omega + tdJ_{\X_p}d\phi|_{\X_p},$$ so that $\omega_{t,p}$ interpolates between $\omega$ and $\omega_{\X_p}$. Writing  $$\beta_p = J_{\X_p}d(\phi|_{\X_p}),$$ by definition $\omega_{p,t} = \omega + t\beta_p$ and $\beta_p$ is a one-form with coefficients in $L^2_{k+1}(M)$.  For each $t\in [0,1]$ we obtain a vector field $u_{p,t}$, with coefficients in $L^2_{k}$, as the $\omega_{p,t}$-dual of $-\beta_p$, so that $$\frac{d}{dt}\omega_{t,p} = \beta_p = -\iota_{u_{p,t}} \omega_{t,p} = \L_{u_{p,t}}\omega_{p,t}.$$ Taking the time-one flow of $u_{1,p}$ gives an $L^2_{k}$-diffeomorphism $f_p$ which satisfies $$f_p^*\omega_p = \omega.$$ In particular $f_p$ induces an isometry $$f_p: (M,J_{\X_p},\omega_{\X_p}) \cong (M,f_p^*J_{\X_p},\omega),$$ where $f_p^*J_{\X_p}$ is an almost complex structure with $L^2_{k}$-coefficients. 

Define the space $\J^2_k$ to be the space of $L^2_k$-almost complex structures on $(M,\omega)$, which admits an action of $\H^2_{k+1}$, the Sobolev completion of $\H$. Performing this process for all $p\in \hat B$, we obtain for each $p\in \hat B$ a new complex structure $f_p^*J_{\X_p} \in \J^2_k$, so we may view the function $\phi$ on $\X$ defining the change in relative almost K\"ahler metric as inducing a smooth map which we write \begin{equation}\label{eq:perturbedphi}\Phi: \hat B \to  \J^2_k,\end{equation} which is a perturbation of the initial map $\Psi$.  After perhaps shrinking $B$, we may assuming $\Phi$ is an embedding, as the initial map $\Psi$ is an embedding. Giving both sides their natural $H_X$-actions (viewing $H_X\subset \H^2_{k+1}$), the constructions imply that $\Phi$ is further a $K$-equivariant map. For each $p\in B$, if $(M,J_{\X_p},\omega_{\X_p})$ is cscK, then since $$f_p: (M,J_{\X_p},\omega_{\X_p}) \cong (M,f_p^*J_{\X_p},\omega)$$ is an isometry, it remains true that $ (M,f_p^*J_{\X_p},\omega)$ is a cscK manifold. Note that by elliptic regularity, such an almost complex structure is in fact smooth (or, equivalently, the original $\phi|_{\X_p}$ must be smooth). 

\begin{remark}\label{rmk:completeness}
We will later use the following property of the Kuranishi space, which is a form of completeness: there is an open neighbourhood $U$ of $\Phi(0)$ in $\J^{2,\inte}_k$ such that if $J'\in U$, then there is a $p\in B$ such that $(\X_p,\alpha_{\X_p}) \cong (M,J',[\omega])$.
\end{remark}

We will also require a more detailed understanding of Ortu's construction, for which we restrict to the integrable locus $\X \to B$. The relatively K\"ahler metric $\omega_{\X}$ induces a Hermitian metric on the relative tangent bundle $T_{\X/B}$, which in turn induces a metric on the relative anti-canonical class $-K_{\X/B}$ (when $B$ is singular, one may perform a similar process first on $\hat B$ and then restrict). Let $\rho \in c_1(-K_{\X/B})$ be the curvature of this Hermitian metric, so that $\rho$ restricts to the Ricci curvature of $\omega_b$ over each fibre $\X_b$ (which again may be constructed over $\hat B$ and restricted). 

Set $$\hat S = n\frac{\int_{\X_p} \Ric \omega_{\X_p} \wedge \omega_{\X_p}^{n-1}}{\int_{\X_p} \omega_{\X_p}^n} $$ to be the fibrewise average scalar curvature, calculated on any fibre of $\X\to B$, where $n$ is the dimension of  $\X_p$. This is independent of relatively K\"ahler $\omega_{\X} \in \alpha_{\X}$, as the average scalar curvature is independent of the fibre. Indeed, the function sending $p\in B$ to $\int_{\X_p} \Ric \omega_{\X_p} \wedge \omega_{\X_p}^{n-1}$ is the pushforward (or fibre integral) of the closed $(n,n)$-form $\rho\wedge\omega_{\X}$ over $\X\to B$, so is a closed $0$-form on the base $B$, but closed forms on $B$ are simply constants. 

Define a closed $(1,1)$-form by the fibre integral (or direct image) \begin{equation}\label{initial-WP-formula} \Omega_X =  \int_{\X/B}\left(\frac{\hat S}{n+1} \omega_{\X}^{n+1} - \rho \wedge \omega_{\X}^n \right).\end{equation} This is then a $H_X$-invariant K\"ahler metric on $B$ (perhaps after shrinking $B$), and Ortu's construction involves a moment map  \begin{equation}\label{eqn:momentmap}\mu_X: B\to \mfh_X^*\end{equation} for $\mfh_X$ the Lie algebra of $H_X$, constructed through \cite[Theorem 6.4]{DH}, such that $\mu_X(p)$ vanishes if and only if $(\X_p,\omega_{\X_p})$ is cscK. We note that $\Omega_X$ is defined only on the integrable locus $B$, but Ortu shows that both $\Omega_X$ and the moment map extend to $\hat B$.

\section{Construction of the moduli space}\label{section:construction}

The construction of the moduli space consists of two main steps. The first is to construct a local moduli space around a fixed K\"ahler manifold. The proof of this involves strengthening aspects of the deformation theory of cscK manifolds, and using Stein GIT to construct a local quotient. The second step is to glue together the local GIT quotients, for which the main tool is the work of Chen--Sun \cite{CS}.

\subsection{Stabiliser preservation} 

A key step in the construction of the moduli space is to give an appropriate understanding of the stabilisers of the points in the Kuranishi space. It will only be necessary for us to consider the polystable points, which have reductive stabiliser. From the deformation theorem of Theorem \ref{ortu}, we obtain ball $\hat B\subset \tilde H^1$ of interest. Here we identify the stabilisers of certain points in an even smaller ball by using the work of Inoue \cite{EI}. Results such as these are often called ``stabiliser preservation'' results in the stacks literature.

We consider the Kuranishi space around our initial $(X,\omega)$ and follow the notation of Section \ref{sec:deformations}, so that we are primarily interested in the map $$\Phi: \hat B \to \J^2_k$$ of Equation \eqref{eq:perturbedphi}. The finite dimensional compact group $H_X$ acts freely on $\hat B \times\H^2_{k+1} $, hence the quotient
$$ (\hat B \times \H^2_{k+1} )/H_X
$$ 
is a Banach manifold.  
\begin{proposition}\label{EI-injectivity}\cite[Proposition 3.11]{EI}
The map
$$
F: \hat (\hat B \times \H^2_{k+1} )/H_X \to \J^2_k: (p,h) \mapsto h^*\Phi(p)  
$$
is injective for a sufficiently small neighbourhood of the origin in $\hat B\subset \tilde H^1$. 
\end{proposition}

The proof uses the implicit function theorem.  We emphasise that the injectivity occurs for a small neighbourhood of the origin in $\tilde H^1$, but for \emph{all} of $\H^2_{k+1}$. We replace the ball $\hat B$ by the resulting ball, and similarly shrink $B$, so that on $B$ the map remains injective.

The following is proven in essentially the same way as in the work of Inoue \cite[Corollary 3.14]{EI}.

\begin{corollary}\label{cor:stb} 
For any $p \in B$ such that $(X_p,\alpha_p)$ admits a cscK metric, the stabiliser $\H_{\Phi(p)}$ of $\Phi(p)$ in $\H^2_k$ is the stabiliser group $(H_{X})_p\subset H_X$.\end{corollary}

\begin{proof}
We pick a $p \in B \subset \hat B$ such that $\Phi(p)$ is cscK. Since the embedding $\Phi: \hat B_{\gamma} \to \J^2_k$ is $H_X$-equivariant, we have $(H_{X})_p \subset \H_{\Phi(p)}$. Now for $h \in \H_{\Phi(p)}$ we have 
$$F(p,h)=h^*\Phi(p)=\Phi(p)=F(p,\id),$$ and by the injectivity of $F$ we conclude that $$(p,h)=(p,\id),$$ thus $h \in H_X \cap \H_{\Phi(p)}=(H_{X})_p$ because $\Phi$ is also injective.\end{proof}

Thus we may, and do, assume that in $B$, the stabiliser preserving property proven above holds.

\subsection{Orbit structure}

The $G_X$-orbits of $B$ under the automorphism group action on the Kuranishi space parametrise isomorphic polarised manifolds. Here we prove a converse for polystable orbits, using the uniqueness property of cscK metrics and the stabiliser preservation, inspired by the work of Inoue \cite[Corollary 3.12, Proposition 4.10]{EI}.

\begin{proposition}\label{orbits} Consider a polystable point $p \in C$ with corresponding manifold $(\X_p,\alpha_{\X_p})$. If $q$ is polystable and the manifold $(\X_q,\alpha_{\X_q})$ corresponding to $q \in C$ is isomorphic to $(\X_p,\alpha_{\X_p})$, then $q \in G_X.p$.
\end{proposition}

\begin{proof}

The required statement is independent of replacing $p,q$ by points in their respective orbits, hence we may replace $p,q$ with the points corresponding to cscK manifolds in their orbits. 

The embedding $\Phi: B \hookrightarrow \J^2_k$ gives $L^2_k$-complex structures associated to $p,q$ denoted $J_p,J_q$, with $(M,\omega,J_p)$ and $(M,\omega,J_q)$ both cscK by hypothesis. We thus obtain K\"ahler manifolds $(M,\omega,J_p)$ and $(M,\omega,J_q)$ which are cscK. Thus by elliptic regularity (e.g.  \cite[Theorem 6.3]{FS}), the almost complex structures $J_p$ and $J_q$  are actually smooth. 

Since $(M,J_p)$ and $(M,J_q)$ are biholomorphic, there is a smooth map $m: M \to M$ such that $m^*J_q = J_p$. Then by the uniqueness of cscK metrics stated as Theorem \ref{thm:uniqueness} $(ii)$, there is a $g \in \Aut(M,J_p)$ such that $g^*m^*\omega =  \omega$, so that also $J_p = g^*m^*J_q$.  Thus $(m \circ g)^* \omega = \omega$ and $( m \circ g )^*J_q = J_p$, so that in the notation of Proposition \ref{EI-injectivity} $$F(p,\id) = F(q,m\circ g).$$ The injectivity of the map $F$ provided by Proposition \ref{EI-injectivity} implies $m\circ g \in H_X \subset G_X$, and so finally $q \in G_X.p$ as required.\end{proof}

\subsection{Taking the quotient} As the ball $B$ only admits a local $G_X$-action rather than a genuine action, we cannot take the GIT quotient directly. To rectify this, we first replace $C\subset B$ with a Stein neighbourhood $W'_X \subset C$ of the origin in $C\subset \tilde H^1$, and consider the orbit $G.W'_X \subset \tilde H^1$. This is a complex space with a genuine $G$-action, which we will denote $$W_X=G.W'_X \subset \tilde H^1.$$ The  following, due to Heinzner (combining the Extension Lemma on p.\ 636 and the Complexification Theorem on p.\ 631 of \cite{PH}) is the key observation which  allows us to apply Stein GIT.

\begin{proposition} $W_X$ is a Stein space. \end{proposition}

 In order to understand the quotient $W_X \git G_X$, it is important that the GIT quotient is essentially a local procedure in $W_X$, in the following sense.

\begin{lemma} A point in $C$ is polystable with respect to the local $G_X$-action on $B$ if and only if it is polystable with respect to the $G_X$-action on $W_X$.  \end{lemma}

\begin{proof} 

It is clear that for $p \in C$, the orbit $G.p$ is closed if and only if the orbit is closed in $W_X$, which is precisely the statement of the Lemma.
\end{proof}

We now take the quotient of $W_X$ by $G$, which we denote by $$\pi: W_X \to \M_X = W_X \git G_X.$$ Since all orbits in $W_X$ intersect $B$, we obtain the following, which shows that $\M_X$ has a reasonable universal property, and again proves a locality statement for GIT quotients.

\begin{lemma}\label{shrinking} There is a neighbourhood $\B \subset \M_X$ of $\pi(0)$ such that $ \B \subset \pi(C)$. In particular for all $b \in \B$, $$ \pi^{-1}(b) \cap C \neq \varnothing.$$\end{lemma}

This essentially says that the local structure of $\M_X$ is determined entirely by the local action on $B$. We now replace $\M_X$ with $\B$, so that each point of $\M_X$ (by our replacement) corresponds exactly to a K\"ahler manifold which admits a cscK metric and which is a sufficiently small deformation of $(X,\alpha)$.

\subsection{Gluing the local moduli spaces} We now complete the proof of our first main result, namely the existence of a moduli space of cscK manifolds. So far, to each cscK manifold $(X,\alpha_X)$ we have associated a complex space which is a local moduli space $\M_X$ of cscK manifolds around $(X,\alpha_X)$. That is, points of $\M_X$ correspond to cscK manifolds, and by Remark \ref{rmk:completeness}, any sufficiently small cscK deformation of $(X,\alpha_X)$ is represented by a unique point of $\M_X$. What we show next is that one can glue these spaces together to form the moduli space globally. 

We take a point $p_Y \in \M_X$ corresponding to a K\"ahler manifold $(Y,\alpha_Y)$, with corresponding local moduli space $\M_Y = W_Y \git G_Y$, where $G_Y = \Aut(Y,\alpha_Y)$. To glue the local moduli spaces, we give a canonical isomorphism from a neighbourhood $U_Y\subset \M_X$ around $p_Y$ to a neighbourhood $V_Y \subset \M_Y$.  

\begin{proposition}\label{localgluing} There exist open neighbourhoods $U_Y\subset \M_X$ around $p_Y$ and $V_Y \subset \M_Y$ such that $U_Y$ and $V_Y$ are canonically isomorphic. \end{proposition} 

\begin{proof}

Since $W_X$ is (locally) a Kuranishi space for $(X,\alpha_X)$, it is a complete deformation space for small deformations of $X$, and so we may assume that it is complete for $(Y,\alpha_Y)$. In the argument we shall often shrink $W_X,W_Y$, which by  Proposition \ref{shrinking} does not affect the local structure of the quotient. Thus, perhaps after shrinking $W_X$, there is a map $$\sigma: W_X\to W_Y.$$ We compose this map with the quotient map under the $G_Y$-action to obtain a map $$\nu: W_X \to \M_Y = W_Y \git G_Y.$$ 

We wish to obtain a map locally from $\M_X = W_X\git G_X$ to $\M_Y$. To obtain such a map, since the GIT quotient is categorical, it is enough to show that $\nu$ is $G_{X}$-invariant. Here we note that although $W_X$ has been shrunk hence may not be $G_X$-invariant itself, one can always take its orbit and extend $\nu$ to a holomorphic map using that $G_X$ acts by biholomorphisms (by the construction of the GIT quotient such a $G_X$-invariant function will indeed descend to $\M_X$). We thus take two points $p,q \in W_X$ lying in the same $G_{X}$-orbit. To prove the required equivariance, we consider two distinct cases:

\indent \emph{Case (i):} The orbit $G_X.p$ is polystable. Then the manifolds corresponding to $p$ and $q$ in the Kuranishi family over $W_X$ are isomorphic as they lie in the same $G_X$-orbit, and admit a cscK metric. Thus the manifolds appearing as fibres in the Kuranishi family over $W_Y$ corresponding to $\sigma(p)$ and $\sigma(q)$ are also isomorphic and cscK. We cannot conclude from Theorem \ref{ortu} that the orbits of $\sigma(p),\sigma(q)$ are polystable with respect to the $G_Y$-action on $W_Y$, however Theorem \ref{ortu} does imply that the polystable degenerations of $\sigma(p),\sigma(q)$ must represent isomorphic manifolds. These polystable degenerations must then lie in the same $G_Y$-orbit by Corollary \ref{cor:stb}. Sinces GIT quotients map a semistable orbit to the same point as its polystable degeneration, this shows that $\nu(p) = \nu(q)$, as desired.

\indent \emph{Case (ii):} The orbit $G_X.p$ is strictly semistable. Hence $q$ is also semistable. Since $p,q$ are in the same $G_X$-orbit, their corresponding polarised manifolds are isomorphic. Again it follows that the fibres of the family over $W_Y$ corresponding to $\sigma(p)$ and $\sigma(p)$ are isomorphic. Because $\sigma(p)$ and $\sigma(q)$ are only GIT-semistable in this case, we cannot conclude that they lie in the same $G_Y$-orbit. A semistable orbit admits a degeneration to a cscK manifold in the sense of Theorem \ref{thm:CS} by part of Theorem \ref{ortu}. But applying the uniqueness result of Chen--Sun, namely Theorem \ref{thm:CS}, we know that the polystable degenerations of $\sigma(p)$ and $\sigma(q)$ must then represent the same cscK manifold, giving again that  $\nu(p)=\nu(q)$.

Since each small cscK deformation is represented in the quotients $\M_X, \M_Y$ exactly once, the induced map $\M_X \to \M_Y$ we have constructed must be injective. Moreover, as the Kuranishi family for $X$ is complete for $Y$, this map is also locally surjective, ultimately by Remark \ref{rmk:completeness}. We cannot conclude the proof in general at this point, as a bijective holomorphic map between complex spaces is not necessarily a biholomorphism (consider for example the normalisation of a cuspidal curve). 

We instead proceed by using that a Kuranishi space induces a complete deformation for nearby points. This allows us to construct a (local) map $\sigma': W_Y \to W_X$, such that the universal family over $W_X$ pulls back to that over $W_X$. We compose this with the quotient map $W_X \to \M_X$ to obtain a map $\nu': W_Y \to \M_X$. A similar argument to the above shows that $\nu'$ is $G_Y$-invariant, hence we obtain a holomorphic map $\M_Y \to \M_X$ which is clearly an inverse to the map constructed above. This gives the local isomorphism as desired.\end{proof}

We use this to prove our main result.

\begin{corollary}\label{construction} There exists a Hausdorff complex space $\M$ which is a moduli space of polarised manifolds which admit a cscK metric. \end{corollary}

\begin{proof} 

We endow the natural symplectic quotient with the structure of a complex space, using Proposition \ref{localgluing}. Consider the locus $$\left(S(\cdot ) - \hat S\right)^{-1}(0)\subset \J^{2,\inte}_k,$$ parametrising integrable almost complex structures with constant scalar curvature. Note by elliptic regularity, this space is actually independent of $k$. Here for $J\in \J^{\inte}_k$ an integrable almost complex structure on $M$ compatible with $\omega$, $S(J)$ denotes the scalar curvature of the K\"ahler metric $\omega$ on $(M,J)$. This space admits an action of $\H^2_{k+1}$, and the quotient $$\M = \left(S(\cdot) - \hat S\right)^{-1}(0)/\H^2_{k+1}$$ is a locally compact Hausdorff topological space (as in Inoue \cite[Section 3.3]{EI}).

Following Section \ref{sec:deformations}, we consider a cscK manfiold $(X,\alpha)$, so that the Kuranishi theory produces a $H_X$-equivariant (not necessarily holomorphic) injection $$\Phi: B\to  \J^{2,\inte}_k.$$ The locus of cscK metrics in $B$, through this injection, is contained in the polystable locus in $B$. We claim that $$\left(\left(S(\cdot ) - \hat S\right)^{-1}(0) \cap \Phi(B)\right)/H_X \cong \M_X.$$ Indeed, the map from the cscK locus of $B$ (namely the locus of $p\in B$ for which $\Phi(p)$ is cscK) to $\M_X$ is a $H_X$-invariant continuous map, and induces a continuous bijection  $$\left(\left(S(\cdot ) - \hat S\right)^{-1}(0) \cap \Phi(B)\right)/H_X \to \M_X.$$ Since both spaces are locally compact and Hausdorff, this induced map is a homeomorphism. 

In particular, we obtain an induced continuous map $$\Phi_X: \M_X \to \M.$$ We claim that $\Phi_X$ is a homeomorphism onto its image, following an argument of Inoue \cite[Corollary 3.12]{EI}.

Firstly, it follows from Proposition \ref{orbits} that the space $\M$ is a \emph{topological} moduli space, parametrising cscK manifolds, ultimately by the uniqueness of cscK metrics employed there. Thus since $\M_X$ also parametrises cscK manfiolds and the map $\Phi_X$ is compatible with this identification, the map $\Phi_X$ is injective. Its surjectivity  follows from completeness of the Kuranishi space, much as in the proof of Proposition \ref{localgluing}. Indeed, by Remark \ref{rmk:completeness} there is a neighbourhood around $\Phi(0)$ in $ \J^{2,\inte}_k$ such that any almost complex structure in this neighbourhood is biholomorphic to a complex structure corresponding to a point of $B$, and if any such almost complex structure admits a cscK metric, then a cscK metric (or more precisely, a cscK almost complex structure) is obtained as the image of a point in $B$. Thus $\Phi_X$ is again a continuous bijection between locally compact Hausdorff topological spaces, so a homeomorphism. In particular, the image $\Phi_X(\M_X)$ is open in $\M$.

We next use the maps $\Phi_X$ to endow the topological space $\M$ with the structure of a locally ringed space. Indeed, letting $\O_{\M_X}$ denote the sheaf of holomorphic functions on $\M_X$, the homeomorphism $\Phi_X$ endows the subset $\Phi_X(\M_X)$ with the sheaf $\Phi_{X*}(\O_{\M_X})$. For $X,Y$ with $\M_X\cap \M_Y \neq \emptyset$, Proposition \ref{localgluing} then gives an isomorphism between these sheaves on $\M_X\cap \M_Y$. We thus obtain a cover of the topological space $\M$, together with a sheaf on each chart, such that on triple intersections $\M_X\cap \M_Y \cap \M_Z$ the following diagram commutes:

\[
\begin{tikzcd}
 & \O_{\M_Y}|_{\M_X\cap \M_Y \cap \M_Z} \arrow[dr," "] \\
\O_{\M_X}|_{\M_X\cap \M_Y \cap \M_Z} \arrow[ur," "] \arrow[rr," "] && \O_{\M_Z}|_{\M_X\cap \M_Y \cap \M_Z}.
\end{tikzcd}
\]
This thus produces a sheaf $\O_{\M}$ on $\M$ whose restriction to $\M_X$ is isomorphic to $\O_{\M_X}$ \cite[Lemma 6.33.2]{stacks-project}.\color{black}

We have constructed a locally ringed space $(\M,\O_{\M})$, and the final claim is that this gives $\M$ the structure of a complex space, for which we need to produce---around each point---a local biholomorphism with an open neighbourhood of the origin in $\C^r/V(I)$ for a finitely generated ideal $I$ of holomorphic functions. This is provided by the construction, as $\M_X$ is constructed as a Stein GIT quotient, so $\M_X$ is Stein and $(\M,\O_{\M})$ is a Hausdorff complex space. \end{proof}

\subsection{Degenerations of cscK manifolds}

An important aspect of Theorem \ref{mainthm} is that the moduli space is \emph{Hausdorff}. Translating this statement into a statement about families of cscK manifolds, this states that sequential limits of cscK manifolds are actually unique. Here we prove some more precise statements along these lines.

\begin{definition}\label{def:semistable}We say that a polarised manifold $(X,\alpha_X)$ is \emph{analytically K-semistable} if there exists a polarised family $(\X,\alpha_{\X}) \to \C$ with $(\X_t,\alpha_{\X_t}) \cong (X,\alpha_X)$ for all $t\neq 0$ and such that $(\X_0,\alpha_{\X_0})$ is a polarised manifold which admits a cscK metric. \end{definition}

The following is a simple consequence of Ortu's Theorem \ref{ortu}; we briefly give the argument, as we will use the statement several times later. 

\begin{lemma}\label{lemma:open} Analytic semistability is an open condition in families of polarised manifolds. \end{lemma}

Here, as throughout, we mean openness in the analytic topology.

\begin{proof}Let $(\Y,\alpha_{\Y}) \to B$ be a family of polarised manifolds with $(\Y_q,\alpha_{\Y_q})$ analytically K-semistable for some $p \in B$. Let $(Z,\alpha_{Z})$ be the cscK degeneration of $(\Y_q,\alpha_{\Y_q})$, and consider the Kuranishi space $W$ of the cscK manifold $(Z,\alpha_{Z})$ constructed in Theorem \ref{ortu}. Then $W$ is a complete deformation space in a neighbourhood of $(Z,\alpha_{Z})$, and hence is complete for $(\Y_q,\alpha_{\Y_q})$ as $(\Y_q,\alpha_{\Y_q})$ admits a degeneration to $(Z,\alpha_{Z})$.

Since the statement is local, we can thus assume there is a map $p: B \to W$ such that $(\Y,\alpha_{\Y}) \to B$ is the pullback of the universal family over $W$. In particular, there are balls $W_{\delta} \subset W_{\epsilon} \subset W$ with $p(\Y) \subset W_{\delta}$ such that for each point $w \in W$, either $p$ is polystable or the polystable degeneration of $w$ is contained in $W_{\epsilon}$. From Theorem \ref{ortu}, the manifolds associated to these polystable degenerations admit a cscK metric, and hence each point in $p^{-1}(W_{\delta})$ is analytically K-semistable as claimed. \end{proof}

Analytically K-semistable manifolds are the analogue of semistable orbits in GIT, and it follows from Thereom \ref{ortu} and \cite{SD-lower,RD2} that they are K-semistable. The above is the analogue of semistability being an open condition in GIT. A standard result in GIT is that the closure of a semistable orbit intersects a \emph{unique} polystable orbit; the corresponding statement for analytically K-semistable manifolds is Chen--Sun's Theorem \ref{thm:CS}. In GIT there are also more precise statements about the behaviour of polystable and semistable orbits in families. The analogue in our situation is the following.

\begin{corollary}\label{separated-moduli}

Let $C$ be a (not necessarily compact) Riemann surface with $0 \in C$ a point. Suppose $$\xymatrix@C=14pt@R=16pt{
&(\Y,\alpha_{\Y}) \ar[dr]&& \mathcal (\tilde \Y,\alpha_{\tilde \Y}) \ar[dl] \\
 &   & C }
$$ are two polarised families such that $({\Y},\alpha_{\Y})|_{C\backslash\{0\}} \cong  (\tilde \Y,\alpha_{\tilde \Y})|_{C\backslash\{0\}}.$

\begin{enumerate}[(i)]
\item (Semistable case) If  $(\Y_0,\alpha_{\Y_0})$ and $(\tilde \Y_0,\alpha_{\tilde \Y_0})$ are analytically K-semistable, then they degenerate to the same cscK manifold. 
\item (Polystable case) If  $(\Y_0,\alpha_{\Y_0})$ and $(\tilde \Y_0,\alpha_{\tilde \Y_0})$ admit a cscK metric, then $(\Y_0,\alpha_{\Y_0})\cong (\tilde \Y_0,\alpha_{\tilde \Y_0})$.  
\item (Stable case)  If $(\Y_0,\alpha_{\Y_0})$ admits a cscK metric, $\Aut(\Y_0,\alpha_{\Y_0})$ is discrete and $(\tilde \Y_0,\alpha_{\tilde \Y_0})$ is analytically K-semistable, then $(\Y_0,\alpha_{\Y_0})\cong  (\tilde \Y_0,\alpha_{\tilde \Y_0})$.

\end{enumerate}

\end{corollary}

\begin{proof} Since each of the statements is local around $0 \in C$ (as we allow $C$ to be shrunk in $(iii)$), we freely replace $C$ with a neighbourhood of $0$. Since analytic semistability is an open condition by Lemma \ref{lemma:open}, we can therefore assume all fibres are analytically K-semistable.

As each fibre over $C_t$ is analytically K-semistable, associated to each $t$ is a point $q_t\in \M$ corresponding  to the cscK degeneration of $(\Y_t,\alpha_{\Y_t})$. The main point of the proofs of the claims is that the $q_t$ depend continuously on $t$, namely that the induced map $$\nu: C \to \M$$ is continuous.  To prove this, we may consider the Kuranishi space $(\X,\alpha_{\X},\omega_{\X}) \to B$ of the cscK degeneration of $(\Y_0,\alpha_{\Y_0})$. By completeness, the family $(\Y,\alpha_{\Y})$ is pulled back from a morphism to $B$, and by Theorem \ref{ortu}, the cscK degeneration of each $p\in B$ corresponds to a point $q\in B$. It then follows from \cite[Proposition 2.4 and Equation (2.1)]{ortu}  (which relies on a result of Duistermaat \cite{flow}) that there is a \emph{continuous} map $B\to B$ such that each $p$ is sent to a corresponding $q \in B$ such that $\omega_{\X_q}$ is cscK. This implies that the points $b_t$ vary continuously with $t$, by construction of the moduli space.

$(i)$ By hypothesis, both $(\Y_0,\alpha_{\Y_0})$ and $(\tilde \Y_0,\alpha_{\tilde \Y_0})$ degenerate to a cscK manifold. Thus we obtain continuous maps $\nu: C\to \M$ and $\tilde \nu: C' \to \M$ by the preceding paragraph, such that $\nu(t) = \tilde \nu(t)$ for $t\neq 0$. Hausdorffness and continuity imply that $\nu(0)=\tilde \nu(0)$, which implies the result.

$(ii)$ This is immediate from $(i)$, together with the fact that a cscK manifold cannot degenerate to another non-isomorphic cscK manifold, by Theorem \ref{thm:CS}.

$(iii)$ By $(i)$, the cscK degeneration of $(\tilde \Y_0,\alpha_{\tilde \Y_0})$ must be $( \Y_0,\alpha_{ \Y_0})$. Consider the Kuranishi space $(\X,\alpha_{\X}) \to B$ of $( \Y_0,\alpha_{ \Y_0})$,  so that $( \Y_0,\alpha_{ \Y_0})$ corresponds to the origin in $B$ and the origin lies in the closure $0\in \overline{G_{\Y_0}}.q$ for a point $q$ corresponding to $(\tilde \Y_0,\alpha_{\tilde \Y_0})$ (since by Theorem \ref{ortu}, the closure of this orbit must contain a cscK manifold and such a manifold is unique). Since $\overline{G_{\Y_0}}$ is actually finite by assumption, it follows that $(\tilde \Y_0,\alpha_{\tilde \Y_0}) \cong( \Y_0,\alpha_{ \Y_0})$.  \end{proof}

\section{K\"ahler geometry of the moduli space}\label{sec-kahler}

\subsection{The Weil--Petersson metric}\label{sec-kahler2}

We next construct a K\"ahler metric on the moduli space, beginning with some general preliminary results around holomorphic submersions.

We first consider $\pi: \X \to B$ be a proper holomorphic submersion of relative dimension $n$, with $B$ a complex space and with $\alpha_{\X}$ a relatively K\"ahler class. Similarly to Section \ref{sec:deformations}, for $\omega_{\X}$ relatively K\"ahler, we obtain $\rho\in c_1(-K_{\X/B})$ and a (not-necessarily-positive) form \begin{equation}\label{WP2}\Omega_X =  \int_{\X/B}\left(\frac{\hat S}{n+1}\omega_{\X}^{n+1} - \rho \wedge \omega_{\X}^n \right), \end{equation} with notation as in Equation \eqref{initial-WP-formula}, so in particular $\hat S$ denotes the fibrewise average scalar curvature. In applications, when singular, $B$ will be taken to be the integrable locus inside the Kuranishi space, so one may construct the fibre integral by first constructing it on the ball $\hat B$ and restricting, as in Section \ref{sec:deformations}.

In order to understand the dependence of this fibre integral on the choice of $\omega_{\X}$, we recall the following important functional in K\"ahler geometry.

\begin{definition} Let $(Y,\alpha_Y)$ be a compact K\"ahler manifold of dimension $n$, and fix a reference K\"ahler metric $\omega_Y \in \alpha_Y$. Denote by $\scH_{\alpha_Y}$ the set of K\"ahler potentials in $\alpha_Y$ with respect to $\omega_Y$. For any $\omega_Y+\ddb \phi \in \scH_{\alpha_Y}$, let $\phi_t$ be a path with $\phi_0 = 0$ and $\phi_1=\phi$. Denote the associated K\"ahler metrics  along the path by $\omega_{Y,t} = \omega_Y + \ddb \phi_t$. The \emph{Mabuchi functional} is the functional $\scK_{\omega_Y}: \scH_{\alpha_Y} \to \R$ defined such that $$\frac{d}{dt}\scK_{\omega_Y}(\phi_t) = -\int_X\dot \phi_t (S(\omega_{Y,t}) - \hat S_Y)\omega_{Y,t}^n$$ and $\scK_{\omega_Y}(0)=0$.
\end{definition}

It is a classical result of Mabuchi that the definition is independent of choice of path \cite{mabuchi-symplectic}, producing a well-defined functional $\scK_{\omega_Y}: \scH_{\alpha_Y} \to \R$. The other key property that we will need, which is deeper, is that if $\omega_Y, \tilde \omega_Y \in \alpha_Y$ are cscK, then $\scK_{\omega_Y}(\tilde \omega_Y) = 0$ \cite{donaldson-embeddings-ii, BB}. More generally, and with respect to an arbitrary basepoint, any cscK metric achieves the absolute minimum of the Mabuchi functional. It follows that the value taken by the Mabuchi functional at an arbitary cscK metric is independent of cscK metric.

Returning to our submersion $\pi: \X \to B$, by restriction for each $p\in B$ the relatively K\"ahler metric $\omega_{\X}$ induces a K\"ahler metric $\omega_{\X_p} = \omega|_{\X_p}$, and hence one obtains a functional $\scK_{\omega_{\X_b}}: \scH_{\alpha_p} \to \R$. Thus if $\tilde \omega_{X}$ is another relatively K\"ahler metric,  one obtains a function $\scK_{\omega_X}(\tilde \omega_X): B \to \R$ defined by $$ \scK_{\omega_X}(\tilde \omega_X)(p) = \scK_{\omega_b}(\tilde \omega_p).$$ We also obtain a form $\tilde \Omega_{\X}$ on $B$ produced by the analogous formula to Equation \eqref{WP2}.

The following result is due to Phong-Ross-Sturm in the case $X$ and $B$ are projective \cite[Equation (7.5)]{PRS}, and follows from the techniques of \cite[Section 3]{DRoss} and \cite[Section 4]{ZSD2} in the general K\"ahler setting. 

\begin{proposition}\label{PRS}The fibre integrals $\Omega_{\X}$ and $\tilde \Omega_{\X}$ are related by $$\Omega_{\X}  - \tilde \Omega_{\X} = \ddb \scK_{\omega_{\X}}(\tilde \omega_{\X}).$$\end{proposition}

This will allow us to prove our first result on the behaviour of the forms $\Omega_{\X}$.

\begin{corollary}\label{cor:independence} Let $\omega_X, \tilde \omega_X \in \alpha$ be two relatively K\"ahler metrics which are each cscK on each fibre. Then $$\Omega_{\X}  = \tilde \Omega_{\X} .$$
\end{corollary}

\begin{proof}

The value of the fibre integral over the dense smooth locus of $B$ determines the value of the fibre integral on $B$, so we may assume $B$ is smooth; one could alternatively pull-back to a resolution of singularities. Then to show that $\Omega_{\X}) = \tilde\Omega_{\X})$, by Proposition \ref{PRS} it is enough to show that $\scK_{\omega_{\X}}(\tilde \omega_{\X}) = 0.$ But by definition $$\scK_{\omega_{\X}}(\tilde \omega_{\X})(p) = \scK_{\omega_{\X_p}}(\tilde \omega_{\X_p}) = 0$$ for each $p \in B$, since $\omega_p$ and $\tilde \omega_p$ are both cscK.\end{proof}

We next consider a fixed chart $\M_X$ of the moduli space $\M$ of cscK manifolds around a cscK manifold $(X,\alpha_X)$, such that $\M_X = W_X\git G_X$ with $W_X$ lying inside the Kuranishi space of $(X,\alpha_X)$ and $G_X = \Aut(X,\alpha_X)$. We will allow ourselves to shrink $W_X$ throughout, as this will not affect the arguments. From Section \ref{sec:deformations}, we obtain a K\"ahler metric $\Omega_X$ on $X$ defined as a fibre integral following Equation \eqref{WP2} (with positivity following from the Kuranishi theory of Section \ref{sec:deformations}). This metric $\Omega_X$ induces a K\"ahler metric on the GIT quotient $W_X\git G_X$ which we denote $\eta_X$,  constructed in the following manner (generalising the standard construction of symplectic forms on symplectic quotients to the singular setting), due to Heinzner--Huckleberry--Loose \cite{HHL}. 

Let $$\mu_X: W_X\to \mfh_X^*$$ be the associated moment map described in Equation \eqref{eqn:momentmap}, and consider an open neighbourhood $U$ of $b\in\mu_X^{-1}(0)$ on which we can write $$\ddb \phi_X = \Omega_X.$$ Then $\phi_X|_{\mu^{-1}(0)\cap U}$ is $H_X$-invariant and hence descends to the image of $U$ in $W_X\git G_X$. This induced function $\phi_X$ is strictly plurisubharmonic, and this construction induces a K\"ahler metric $\eta_X$ on $\M_X$ \cite{HHL}. The basic reason this is well-defined is as follows \cite[p.\ 139]{HHL}: if $\tilde \phi_X$ is another local potential, then $\phi_X - \tilde \phi_X$ is harmonic, so $$\phi_X - \tilde \phi_X = \Ima(f) $$ is the imaginary part of a holomorphic function $f$ which is $H_X$-invariant and hence $G_X$-invariant, so descends to a holomorphic function $f'$ on $W_X\git G_X$. Then writing $\psi_X$ and $\tilde \psi_X$ as the induced functions on $W_X\git G_X$, $$\psi_X - \tilde \psi_X = \Ima(f'),$$ meaning their difference is again harmonic.

The K\"ahler metric $\eta_X$ thus produced has continuous local potential, but satisfies further regularity properties: there is a stratification of $\M_X$, essentially by the dimension of the stabiliser, for which $\eta_X$ is actually smooth on each stratum \cite[p.\ 130]{HHL}. 

We next show that these metrics, defined on each chart, agree on intersections of charts and hence glue to a global K\"ahler metric on $\M$. There are two main points in the construction: the first is that the zero locus of the moment map $\mu_X$ described in Equation \eqref{eqn:momentmap} corresponds to cscK metrics (on the universal family); the second is the argument of Corollary \ref{cor:independence}, which uses uniqueness of cscK metrics.

\begin{theorem}\label{gluing-WP}
Suppose $\eta_X$ and $\eta_Y$ are K\"ahler metrics on charts $\M_X$ and $\M_Y$ of $\M$ induced by metrics $\Omega_X$ and $\Omega_Y$ on $W_X$ and $W_Y$. Then $$\eta_X|_{\M_X \cap \M_Y} = \eta_Y|_{\M_X \cap \M_Y},$$ qne hence the collection of $\eta_X$ induce a K\"ahler metric $\eta$ on $\M$.
\end{theorem}

\begin{proof}

Consider a point $p\in \M_X\cap \M_Y$ with $x\in W_X$ satisfying  $\mu_X(x)=0$ and $\pi_X(x) = p$, and similarly  $y\in W_Y$ satisfying $\mu_Y(y)=0$ and $\pi_Y(y)=p$. For each $p$, the form $\eta_X$ is constructed locally  writing $$\eta_X = \ddb \phi_X,$$ restricting $\phi_X$ to $\mu^{-1}_X(0)$ and defining a function on $\M_X$ whose value at $p$ is equal to $\phi_X(x)$, and similarly $\eta_Y$ is induced by a local potential $\phi_Y$ and the moment map $\mu_Y$. 

Since Kuranishi spaces are complete deformation spaces, we obtain a morphism  $$\sigma: W_X\to W_Y$$ with $\sigma(x)=y$, as in Proposition \ref{localgluing}, perhaps after shrinking $W_X$ and $W_Y$, which further pulls back the universal family over $W_Y$  to that over $W_X$. 

We argue on the universal family $\X\to W_X$, hence pull  $\omega_{\Y}$ (the relatively K\"ahler metric on the universal family over $W_Y$) back to a new relatively K\"ahler metric $p^*\omega_{\Y}$ on $\X$, so that $\sigma^*\Omega_{\Y}$ is induced as a fibre integral involving $\sigma^*\omega_{\Y}$ (by definition of such fibre integrals). It follows from Proposition \ref{PRS} that $$\Omega_{\X} - \sigma^*\Omega_{\Y} = \ddb \scK_{\omega_{\X}}(\sigma^*\omega_{\Y}).$$ Thus we may choose our local potentials such that \begin{equation}\label{difference-potentials}\phi_X - \sigma^*\phi_Y =  \scK_{\omega_{\X}}(\sigma^*\omega_{\Y}).\end{equation} Since at $\sigma(x)=y$, both $\omega_{\X}|_{\X_x}$ and $\sigma^*\omega_{\Y}|_{\X_x}$ are cscK metrics, implying the Mabuchi functional $ \scK_{\sigma^*\omega_{\X}}(\sigma^*\omega_{\Y})(x)$ vanishes. It follows that $\phi_X(x)=\phi_Y(y)$, and arguing similarly for points of $\M_X \cap \M_Y$ near $p$ completes the proof. \end{proof}

We call the K\"ahler metric just constructed the \emph{Weil--Petersson metric}. In applications of this theory, one typically has a holomorphic submersion $\pi: (\Y, \alpha_{\Y}) \to B$ and a relatively K\"ahler metric $\omega_{\X}\in \alpha_{\Y}$ which is cscK on each fibre. Thus one obtains a moduli map $\psi: B\to \M$ to the moduli space of cscK manifolds. In this situation, we next prove that the pullback of the Weil--Petersson metric to $B$ may be directly expressed as a fibre integral, which is ultimately a manifestation of the functoriality of the construction of the forms we have constructed on the base of submersions. To state the result, denote the fibre integral of interest on $B$ $$\eta_B = \int_{\Y/B}\left(\frac{\hat S}{n+1}\omega_{\X}^{n+1} - (n+1)\rho \wedge \omega_{\X}^n \right).$$ 

\begin{corollary}   We have an equality of forms $$ \psi^*\eta = \eta_B.$$ 
\end{corollary}

\begin{proof}

The proof is essentially the same as that of Theorem  \ref{gluing-WP}. As the statement is a local one, it is enough to prove the result on an open set $U \subset B$ mapping to a single chart $\M_X$ of $\M$, and we fix points $p\in \M_X$ and $x\in W_X$ satisfying $\pi_X(x)=p$ and $\mu_X(x)=0$, along with $b\in U$ satisfying $\psi(b)=p$. By Kuranishi theory, we obtain a morphism $\sigma: U\to W_X$ such that $\Y$ is the pullback of the universal family $\X$ over $W_X$ with $\sigma(b)=x$, so that the morphism $U \to \M_X$ factoring as $$U \xrightarrow{\sigma} W_X \to \M_X.$$ 

We thus obtain two relatively K\"ahler metrics on $\Y$ (or more precisely its restriction to $U$), $\omega_{\Y}$ on $\Y$ (from the setup) and $\sigma^* \omega_{\X}$ on $\X$ (from Theorem \ref{ortu}), and the difference of the associated forms $\eta_B$ and $\sigma^*\Omega_X$ on $B$ is given by the fibrewise Mabuchi functional: $$\eta_B - \sigma^*\Omega_X = \ddb \scK_{\omega_{\Y}}(\sigma^*\omega_{\X}).$$ We choose local potentials $\phi_{\Y}$ and $\sigma^*\phi_{\X}$ for $\eta_B$ and $\sigma^*\Omega_X$ near $b$ and $x$ respectively such that $$\phi_{\Y} - \sigma^*\phi_{\X} = \ddb \scK_{\omega_{\Y}}(\sigma^* \omega_{\X}).$$ Then since both fibre metrics are cscK, the potentials agree at the point. Arguing in the same manner for nearby points completes the proof.\end{proof}

This also implies that the pullback $p^*\eta$ is smooth. The smoothness statement can be thought of as reflecting the idea that the singularities of the Weil--Petersson metric on $\M$ are precisely caused the lack of a universal family over $\M$. Another consequence is the following.

\begin{corollary} The Weil--Petersson metric $\eta$ on $\M$ agrees with the classical Weil--Petersson metric constructed by Fujiki--Schumacher on the moduli space of cscK manifolds with discrete automorphism group. \end{corollary}

\begin{proof} By \cite[Theorem 7.8]{FS}, Fujiki--Schumacher's Weil--Petersson metric on the moduli space of cscK manifolds with discrete automorphism group can be constructed by using  such fibre integrals,  hence the two constructions agree.\end{proof}

\subsection{The CM line bundle}\label{sec:CM}

We now turn to the projective setting, so that we consider smooth projective varieties together with ample line bundles which  admit a cscK metric. We also consider a fixed connected component of the moduli space $\M$ of cscK manifolds. Our aim is to construct a natural line bundle over $\M$. To do this, we begin with the construction of line bundles on the base of families of polarised manifolds.

Consider a flat family of smooth polarised varieties $\pi:(\X,\H_{\X})\to B$ over a complex space $B$, so that $\H_{\X}$ is relatively ample, and suppose the dimension of each fibre is $n$. By properness, the pushforwards $\pi_*\H_{\X}^k$ are a torsion-free coherent sheaves over $B$ and hence $\det \pi_* \H_{\X}^k$ gives a sequence of line bundles on $B$ for all $k \geq 0$. The Knudsen-Mumford expansion \cite[Theorem 4]{KM} states that for $k \gg 0$, there is an expansion $$\det \pi_* \H_{\X}^k = \lambda_{n+1}^{k \choose n+1}\otimes  \lambda_{n}^{k \choose n}\otimes \hdots \otimes \lambda_0,$$ for $\lambda_j \to B, j=0,\hdots n+1$ a set of line bundles on $B$ independent of $k$. As noted by Fujiki--Schumacher, the Knudsen--Mumford theory applies also to the case that $B$ is a complex space \cite[p.\ 163]{FS}.

\begin{definition} We define the \emph{CM line bundle} on $B$ to be $$\L_{\CM,B}= \lambda_{n+1}^{\hat S + n(n+1)} \otimes \lambda_n^{-2(n+1)},$$ where $\hat S$ is the fibrewise average scalar curvature.
\end{definition}

The CM line bundle, which is a $\Q$-line bundle, was first introduced by Fujiki--Schumacher \cite{FS}, then rediscovered by Tian \cite{GT}. The construction of the CM line bundle is functorial, in the following manner: if $\phi: B' \to B$ is a morphism, $(\X',\H_{\X'}) = (\X,\H_{\X})\times_B B'$ is the pullback family and $\L_{\CM,\X'}$ is the CM line bundle associated to $(\X',\L') \to B'$, then we have $\phi^*\L_{\CM,\X} = \L_{\CM,\X'}$. An important property of the CM line bundle, which we will not use, is that its first Chern class satisfies \begin{equation}\label{CMclass}c_1(\L_{\CM,B}) = \pi_*[\hat S c_1(\H_{\X})^{n+1}+ (n+1)c_1(K_{\X / B}).c_1(\H_{\X})^{n}],\end{equation} as proven by Fujiki--Schumacher \cite[Proposition 10.2]{FS}. This should be compared to the fibre integral formula given in Equation \ref{initial-WP-formula}.

We now return to a chart $\pi_X: W_X \to \M_X$ of the moduli space of cscK polarised manifolds, with $B_X \subset W_X$ the Kuranishi space. As the arguments in the present section are more sensitive to working on $B_X$ rather than $W_X$, we will be precise about which is being used. There is a universal family $(\X,\H_{\X}) \to B_X$, producing the CM line bundle which we denote $$\L_{\CM,X} \to B_X.$$

The orbits of the $G_X= \Aut(X,L)$-action on $B_X$ are all isomorphic polarised varieties, hence for all $g \in G_X $ we have $g^*(\X,\H_{\X}) \cong (\X,\H_{\X})$. It follows that the CM line bundle $\L_{\CM,X} \to B_X$ is a $G_X$-equivariant line bundle, and in particular its transition functions are $G_X$-invariant. The transition functions therefore extend naturally to $W_X$, hence we may extend the CM line bundle to $\L_{\CM,X} \to W_X$ in a canonical manner. From another perspective, using the $G_X$-action we may extend the universal family to a family over $W_X = G_X.B$ and take the associated CM line bundle over $W_X$.

Let $\F_{G_X}$ denote the sheaf on $\M_X$ consisting of $G_X$-invariant sections of $\L_{\CM,X}$: $$\F_{G_X}(U) = (\O(\L_{\CM,X})(\pi_X^{-1}(U)))^{G_X}.$$ By Fujiki \cite{AF}, the K\"ahler metric $\Omega_X$ on $W_X$ is the curvature of a Hermitian metric on $\L_{\CM,X} \to W_X$. Using this, the following result of Sjamaar, proved in the general setting of analytic GIT, then shows that $\F_{G_X}$ is associated to a $\Q$-line bundle.

\begin{theorem}[Sjamaar]\cite[Section 2.2]{RS}\label{sjam} $\F_{G_X}$ is the sheaf of sections of a $\Q$-line bundle $\L_X\to \M_X$. The integer $r_X$ such that $\L^{\otimes r_X}_X$ is a line bundle is given by the number of connected components of $G_X$. Moreover, there is a natural $G_X$-equivariant isomorphism $\pi^*\L_X^{\otimes r_X} \cong \L_{\CM,X}^{\otimes r_X}$.
\end{theorem}

Sjamaar also proves that $\L_X$ is itself the analytic GIT quotient \begin{equation}\label{eqgit} \L_X^{\otimes r_X} = \L^{\otimes r_X}_{\CM}\git G_X,\end{equation} where it is crucial that every point in $W_X$ is semistable \cite[Proposition 2.15]{RS}.   

It follows that locally, the CM line bundle descends to a line bundle $\L_X$ on the chart $\M_X$ of the moduli space $\M$. Here we globalise this construction to produce a line bundle $\L$ on $\M$.

\begin{theorem} \label{CM-thm}

The moduli space $\M$ of projective cscK manifolds admits a $\Q$-line bundle $\L$ such that if $(\X,\H_{\X})\to B$ is any family of cscK manifolds, then the CM line bundle $\L_{\CM,\X} \to B$ is isomorphic to the pullback $\psi^*\L$ through the moduli map $\psi: B \to \M$. 

\end{theorem}

\begin{proof}
We begin by constructing $\L$. From the above, on each chart $\M_X$ we obtain a $\Q$-line bundle $\L_X$ such that $\L_X^{r_X}$ is a line bundle. The integer $r_X$ is bounded in the following manner, as in Fujiki--Schumacher \cite[Lemma 11.8]{FS}. By Matsusaka's big theorem, there is a $k \gg 0$ such that for all smooth polarised varieties $(X,L)$ of fixed Hilbert polynomial, the line bundle $L^{\otimes k}$ is very ample. Thus $\Aut(X,L) \subset GL(H^0(X,L^{\otimes k})^*)$ for all $(X,L)$ with fixed Hilbert polynomial, and hence the number of connected components of $\Aut(X,L)$ is uniformly bounded. Our moduli space $\M$ parametrises polarised varieties with fixed Hilbert polynomial, as this is constant in flat families. From Theorem \ref{sjam},  this produces an integer $r$ such that $\L_X^{\otimes r}$ is a line bundle for all $[X] \in \M$. For notational convenience, we assume $r=1$, as this does not affect our argument.

The construction of $\M$ involves a collection of charts $\M_X$ together with biholomorphisms $\varphi_{XY}: \M_X \to \M_Y$ when $U_{XY} = \M_X\cap\M_Y\neq \varnothing$. We begin by constructing an isomorphism $\alpha_{XY}: \L_X \cong \L_Y$ valid on $U_{XY}$. 

Consider a point $p \in U_{XY}$ with $p = \pi_X(x) = \pi_Y(y)$, where $\pi_X: W_X \to \M_X$, $\pi_Y: W_Y \to \M_Y$ are the natural quotient maps. From Theorem \ref{sjam}, we obtain a $G_X$-equivariant isomorphism $\pi_X^*\L_X \cong \L_{\CM,X}$. 

There is a morphism $\gamma_{XY}: V_X \to V_Y\subset W_Y$, defined in a neighbourhood $V_X$ of $x$, such that the pullback of the universal family over $V_Y$ is isomorphic to the universal family over $V_X$. Over $\gamma_{XY}^{-1}(V_Y)$, the pullback of the CM line bundle $\L_{\CM,Y} \to V_Y$ is canonically isomorphic to the CM line bundle $\L_{\CM,X} \to \gamma_{XY}^{-1}(V_Y) \subset V_X$, namely $$\L_{\CM,X} \cong \gamma_{XY}^*\L_{\CM,Y}.$$ In particular, we obtain bundle morphisms (of bundles over different complex spaces) $$\L_{\CM,X} \to \gamma_{XY}^*\L_{\CM,Y} \to \L_{\CM,Y} \to \L_{Y}, $$ where the last morphism arises from the fact that the GIT quotient $\L_{\CM,Y} \git G_Y$ equals $\L_Y$. 

We show that the composition $\L_{\CM,X} \to \L_Y$ is $G_X$-invariant, from which it will follow that there is a bundle morphism $\alpha_{XY}: \L_X \to \L_Y$, using that $\L_{\CM,X}\git G_X = \L_X$. One way to see this is from the $G_X$-invariant isomorphism $\pi_X^*\L_X \cong \L_{\CM,X}$. Using this isomorphism, it is enough to show that the induced morphism $\pi_X^*\L_X\to \L_Y$ is $G_X$-invariant. But the $G_X$-action on the pullback bundle $\pi_X^*\L_X$ is induced from the $G_X$-action on $V_X$ itself, so the required $G_X$-invariance follows from that of $V_X \to \M_Y$, which is the content of Proposition \ref{localgluing}. For example, on each fibre of $\pi_X$ the pullback $\pi^*\L_X|_{\pi^{-1}(p)} \cong \pi^{-1}(p)\times \C$ is trivial, and the $G_X$-action on $\pi^*\L_X|_{\pi^{-1}(p)}$ is just given by $g.(x,l) = (g.x,l)$ for $x \in \pi^{-1}(p)$.

Thus for each $p$, we obtain a bundle morphism $\alpha_{XY}: \L_X \to \L_Y$ defined in a neighbourhood of $z$. But these morphisms agree on intersections, giving a morphism $\alpha_{XY}: \L_X \to \L_Y$ defined on $U_{XY}$. This bundle morphism must be an isomorphism, since $\L_{\CM,X} \to \gamma_{XY}^*\L_{\CM,Y}$ is a bundle isomorphism.

Let $\xi_X: \L_X \to \M_X\times\C$ and $\xi_Y: \L_Y \to \M_Y\times\C$ be trivialisations. Then \begin{equation}\label{transeqn}\xi_X \circ \alpha^{-1}_{XY}\circ \xi_Y^{-1}: U_{XY}\times\C \to U_{XY}\times\C\end{equation} can be viewed as a function $ \psi_{XY}\in \O^*(U_{XY})$ in the usual manner. We show that the collection of $\psi_{XY}$ define an element of $H^1(\M,\O^*)$, and hence produce a line bundle. 

We first show that $\psi_{XY} = \psi_{YX}^{-1}$. As above, we fix $p = \pi_X(x') = \pi_Y(y')$, and a map $\gamma_{YX}: V'_Y \to V'_X\subset W_X$, defined now in a neighbourhood $V'_Y$ of $y'$ which may be different from $V_Y$. The induced morphism $\L_{\CM,Y} \to \L_Y$ is $G_Y$-invariant as above, and for our purposes it is sufficient to show that $\alpha_{YX} = \alpha_{XY}^{-1}$. We show this at the point $z$ itself, as the argument is the same for any point in a neighbourhood of $z$. We may choose $x,x',y,y'$ in the preimage of $[$ to have closed orbits. By Proposition \ref{localgluing}, we know that $x \in G_X.z'_X$ and $y \in G_Y.y'$. By equivariance of the objects involved, we may therefore assume that $x = x', y = y'$. We consider the open neighbourhood $V_Y \cap (\gamma_{XY} \circ \gamma_{YX})^{-1}(V_Y)$ of $y$, where we have two \emph{isomorphic} universal families. By functoriality of the CM line bundle, for each of these universal families, the CM line bundle is actually the same. It follows that $\L_{\CM,Y} \to (\gamma_{XY} \circ \gamma_{YX})^*\L_{\CM,Y}$ descends to the identity map $id = \alpha_{YX} \circ \alpha_{XY}^{-1}: \L_Y \to \L_Y$, which implies  $\psi_{XY} = \psi_{YX}^{-1}$ as desired. The argument to show $\psi_{XY} \psi_{YZ} \psi_{ZX} = id$ where defined is identical, with slightly more laborious notation. This, therefore, constructs a $\Q$-line bundle $\L \to \M$.

We next show that this construction is independent of chosen charts. By passing to a common refinement one can assume that the charts are actually equal, with transition functions $\psi_{XY}$ and $\psi_{XY}'$ defined on these charts. The argument is essentially the same as that of the previous paragraph. Suppose $\pi_X: W_X \to \M_X$ and $\pi_{X'}: W_{X'} \to \M_X$ are the two local charts. Then, around any $x \in W_X$, one obtains local maps $W_X \to W_{X'}$ under which the CM line bundle on $W_{X'}$ pulls back to a bundle canonically isomorphic to the CM line bundle on $W_X$. Then the previous paragraph shows that the transition functions constructed can be taken to be the same, showing that the construction is indeed independent of chart.
 
Finally, suppose that $(\X,\H_{\X})\to B$ is a family of cscK manifolds. We show that the induced CM line bundle which we denote $\L_{\CM,\X}\to B$ is isomorphic to the pullback $\psi^\L$ through the moduli map $\psi: B \to \M$. The bundle $\L$ is defined through its transition functions, and the transition functions of $\psi^*\L$ are just the pullback of these transition functions. Cover $B$ by open sets $U_{\X_b}$ such that each $U_{\X_b}$ is mapped to the Kuranishi space for the fibre $(\X_b,\H_{\X_b})$. We then have morphisms $$U_{\X_b} \to W_X \to \M_X$$ which is just the definition of the moduli map. By functoriality of the CM line bundle, through these morphisms we obtain a canonical isomorphism $\L_{\CM,\X}|_{U_{\X_b}} \cong \psi^*\L_{\CM}|_{U_{\X_b}}.$ As above, these local isomorphisms glue to a global isomorphism $\L_{\CM,\X} \cong \psi^*\L$ by the same transition function argument. \end{proof}

So far we have constructed a K\"ahler metric $\eta$ on $\M$ and a line bundle $\L\to \M$. As these are both canonical objects, it is natural to ask how the two are related. Our next result shows that the Weil--Petersson metric is actually a Hodge metric, arising as the curvature of a Hermitian metric on $\L$.

\begin{theorem}\label{CM-metric} There is a Hermitian metric on $\L_{\CM}$ whose curvature is the Weil--Petersson metric.  \end{theorem}

\begin{proof} On a given chart $(\X,\H_{\X})\to W_X)$ with induced CM line bundle $\L_{\CM,\X}$, the form $\eta_X$, constructed as a fibre integral on $W_X$, is the curvature of a natural Hermitian metric on $\L_{\CM,X}$. Indeed, the CM line bundle can be viewed as a Deligne pairing $$\L_{\CM,X} = \langle \H_{\X}, \hdots, \H_{\X}\rangle^{\otimes \hat S} \otimes\langle K_{\X/W_X}, \H_{\X},\hdots,\H_{\X}\rangle^{\otimes n+1}.$$ Taking $h_{\X}$ to be a Hermitian metric with curvature $\omega_{\X}$, we obtain a Hermitian metric on the relative canonical class $K_{\X/W_X}$, hence we obtain a Deligne metric $h_{\CM,X}$ on $\L_{\CM,X}.$  We refer to \cite[Section 7]{PRS} for details on the Deligne pairing and Deligne metrics. Thus on each chart, the CM line bundle descends to a line bundle on $\M_X$ with an induced continuous Hermitian metric, whose curvature is the Weil--Petersson form. What we must show is that this Hermitian metric is well-defined, independently of the chosen chart.

We thus consider a trivialising section $s$ of the CM line bundle $\L$ on $\M$ on an open set lying in the intersection $\M_X\cap \M_Y$ of two charts. Then by construction of $\L$, the section corresponds to $G_X$ and $G_Y$-invariant sections $s_X$ and $s_Y$ respectively of $\L_{\CM,X}$ and $\L_{\CM,Y}$, and the value $|s|^2_{h(p)}$ of the Hermitian metric $h$ on $\L$  is computed as $$|s|^2_{h}(p) = |s|_{h_{\CM,X}}^2(x),$$ where $x$ is a zero of the moment map in $W_X$ mapping to $p$. We must show this agrees with the corresponding value computed for $s_Y$, for which we argue as in Theorem \ref{gluing-WP}. Through the induced morphism $\sigma: W_X\to W_Y$, by functoriality of the CM line bundle, we obtain two Hermitian metrics $h_{\CM,X}$ and $\sigma^*h_{\CM,Y}$ on $\L_{\CM,X}$. The Deligne metric satisfies the property that \cite[Equation (7.5)]{PRS} $$h_{\CM,X} = \sigma^*h_{\CM,Y}e^{\scK_{\omega_{\X}}(\sigma^*\omega_{\Y})}$$ with $\omega_{\X}, \omega_{\Y}$ the curvatures of the metrics $h_{\X}, \sigma^*h_{\Y}$ on the respective universal families, which is the crucial general property of Deligne metrics, and so as in Theorem \ref{gluing-WP} since the values $|s|^2_{h}(p)$ are computed at zeroes of the moment maps, the two values agree.

We note that the metric thus constructed is continuous, parallel to the fact that the Weil--Petersson metric is constructed to have continuous potential.\end{proof}

\begin{remark}A slightly weaker statement is that if $(\X,\L)\to B$ is any family of cscK manifolds, with moduli map $\psi: B \to \M$ , then $\psi^*\eta$ is the curvature of a Hermitian metric on the CM line bunde $\L_{\CM,\X}$ on $B$. By our results, both of these pullbacks can be constructed directly on $B$, without using the moduli map. This is then actually easier to prove, and follows from a comparison of the fibre integral formula of Equation \ref{initial-WP-formula} and Equation \eqref{CMclass}. \end{remark}

\begin{remark}

In the projective setting, the results of Sections \ref{sec-kahler2} can be phrased using Deligne metrics and their properties (such as the formula for their curvature), as is essentially a consequence of the argument of Theorem \ref{CM-metric}. Similarly, viewing the CM line bundle as induced from a Deligne pairing, the functoriality properties used in the proof of Theorem \ref{CM-thm} can also be viewed as functoriality properties of Deligne pairings.

\end{remark}

\begin{corollary}\label{kodaira} Every compact complex subspace of $\M$ is projective. \end{corollary}

\begin{proof}

Let $B \subset \M$ be such a subspace. We have shown that $B$ admits a K\"ahler metric $\eta|_B \in c_1(\L)$. If $B$ and $\eta|_B$ were smooth, it would follow by the Kodaira embedding theorem that $B$ is projective. In the singular setting we can appeal to Grauert's embedding theorem for complex spaces \cite[Section 3]{HG}, which again proves projectivity.  Note that the definition of a K\"ahler metric we are using implies the notion used by Grauert, as demonstrated by Sjamaar \cite[Theorem 2.17]{RS}.\end{proof}

\begin{example}\label{example} It is not difficult to construct families of cscK manifolds that are parametrised by a compact complex space. A classical example is given by families of Riemann surfaces of genus at least two, fibred over another Riemann surface of genus at least two, as considered by Fine \cite{JF} in a context related to ours. 

Another example is as follows. Consider the moduli space of (slope) stable vector bundles over a compact Riemann surface endowed with an ample line bundle $(S,L)$ of genus at least two. When degree and rank are coprime, stability coincides with semistability and the moduli space is automatically compact \cite{AB}. Thus one obtains compact families of stable vector bundles over $(S,L)$. A result of Hong implies that, for a fixed stable vector bundle $E$ over $(S,L)$, the projectivisation $\pi: \pr(E)\to S$ admits a cscK metric in the class $k\pi^*c_1(L) + c_1(\O_{\pr(E)}(1))$ for $k \gg 0$ \cite{YH}. It clear from Hong's proof that one can take $k$ to be uniform in families when the family is parametrised by a compact space. Thus for any compact family of stable vector bundles, parametrised by a base $B$ say, one obtains a compact family of cscK manifolds parametrised by $B$. It follows from Corollary \ref{kodaira} that such a $B$ must be projective (of course, this can also be obtained from the vector bundle theory). 

One can obtain examples of compact families of extremal K\"aher manifolds with varying automorphism group by considering the case when the degree and rank of the stable are no longer coprime, allowing strictly polystable vector bundles, and applying later work of Hong and Br\"onnle \cite{YH2,TillBundle}. The extremal condition is more general than the condition of admitting a cscK metric, and in this example it seems likely that this does actually produce compact families of cscK manifolds with varying automorphism group in special cases.

\end{example}

 \bibliography{moduli}
\bibliographystyle{amsplain}

\end{document}